\documentclass[12pt]{amsart}
\usepackage{amsmath}
\usepackage{amsfonts}
\usepackage{amsthm}
\usepackage{amssymb}

\theoremstyle{plain} \numberwithin{equation}{section}
\newtheorem{Theorem}{Theorem}
\newtheorem{Lemma}[Theorem]{Lemma}
\newtheorem{Proposition}[Theorem]{Proposition}

\newtheorem{Remark}[Theorem]{Remark}

\newtheorem{Criterion}[Theorem]{Criterion}

\theoremstyle{remark}

\begin{document}

\title[Divergence of spectral decompositions ]
{Divergence of spectral decompositions of Hill  operators with
two exponential term potentials}

{\author{Plamen Djakov}}


\author{Boris Mityagin}

\address{Sabanci University, Orhanli,
34956 Tuzla, Istanbul, Turkey}
 \email{djakov@sabanciuniv.edu}
\address{Department of Mathematics,
The Ohio State University,
 231 West 18th Ave,
Columbus, OH 43210, USA} \email{mityagin.1@osu.edu}

\begin{abstract}
We consider the Hill operator
$$
Ly = - y^{\prime \prime} + v(x)y, \quad   0 \leq  x \leq \pi,
$$
subject to periodic or antiperiodic boundary conditions ($bc$) with
potentials of the form
$$
v(x) = a e^{-2irx} + b e^{2isx}, \quad a, b \neq 0, \; r,s \in
\mathbb{N}, \; r\neq s.
$$

It is shown that the system of root functions does not contain a
basis in $L^2 ([0,\pi], \mathbb{C})$ if $bc$ are periodic or if $bc$
are antiperiodic and $r, s$ are odd or $r=1$ and $s \geq 3. $
\end{abstract}



\maketitle



\section{Introduction}

We consider the Hill operators $L=L_{Per^\pm}(v)$ with smooth
$\pi$-periodic (complex-valued) potentials $v$
\begin{equation}
\label{1.1}  Ly = - y^{\prime \prime} + v(x)y, \quad   0 \leq  x
\leq \pi,
\end{equation}
subject to periodic ($Per^+$) or antiperiodic ($Per^-$) boundary
conditions:
$$
  Per^\pm : \quad y(\pi) = \pm y(0), \quad y^\prime (\pi) = \pm
y^\prime (0).
  $$
 See basics and details in \cite{MW69}.

If $v $ is real-valued, then $L_{Per^\pm}(v)$
 is a self-adjoint operator with a discrete spectrum.
 The system of its normalized eigenfunctions
   \begin{equation}
  \label{1.2}
  \Phi = \{\varphi_k: \;\;   L \varphi_k = \lambda_k
  \varphi_k, \; \; \|\varphi_k\|=1\}
  \end{equation}
is orthonormal, and the spectral decompositions
  \begin{equation}
  \label{1.3}
  f = \sum_k  \langle  f,\varphi_k  \rangle \varphi_k
  \end{equation}
converge (unconditionally) in $L^2 ([0,\pi])$ for every $f \in L^2
([0,\pi]).$

 If $v$ is a complex-valued potential the picture becomes more
complicated -- see
\cite{Du58,DS71,Ke64,Mi62,Min99,Min06,Sh79,Sh82,Sh83}. In 2006 A.
Makin \cite{Ma06-1,Ma06-2} and the authors \cite[Thm 71]{DM15} gave
the first examples of such potentials that the system of root
functions for periodic or antiperiodic boundary conditions does not
contain a basis in $L^2([0, \pi])$ even though there all but finitely
many eigenvalues are simple.

It is well known that the spectra of the operators $L_{Per^\pm}$ are
discrete, and the following localization formulas hold (see, for
example, \cite[Prop~1]{DM10}):
\begin{equation}
  \label{2.2}
 Sp \, (L_{Per^\pm} ) \subset \Pi_N \cup \bigcup_{n>N, \,
 n\in\Gamma^{\pm}} D_n,    \quad \# \{Sp\,(L_{Per^\pm} )\cap D_n \} =2,
\end{equation}
where $D_n =  \{z: \, |z-n^2| < 1\}, \; \Gamma^{+}= 2\mathbb{N}, \;
\Gamma^{-} =2\mathbb{N} -1, \; N=N(v),$
\begin{equation}
  \label{2.3}
\Pi_N = \{z=x+iy \in \mathbb{C}: \; |x| < (N+1/2)^2, \; |y| < N \}.
\end{equation}
In either case the spectral block decompositions
\begin{equation}
  \label{2.7}
g = S_N g +\sum_{n>N, \,  n\in\Gamma_{\pm}}  P_n  g ,\quad \forall
\,g \in L^2 ([0,\pi]),
\end{equation}
where
\begin{equation}
  \label{2.8}
S_N = \frac{1}{2 \pi i} \int_{\partial \Pi_N} (z-L_{Per^\pm})^{-1}
dz, \quad P_n = \frac{1}{2 \pi i} \int_{\partial D_n}
(z-L_{Per^\pm})^{-1} dz,
\end{equation}
converge unconditionally in $L^2 ([0,\pi]).$   This is true even if
the $\pi$-periodic potential $v$ is singular, i.e., $v \in
H^{-1}_{loc} (\mathbb{R}),$  as A. Savchuk and A. Shkalikov showed
in \cite{SS03}. An alternative proof is given in \cite{DM19}.

The unconditional convergence of decompositions (\ref{2.7}) implies
that for every set $\Delta $ (finite or infinite) of even (or odd)
integers $n>N$ the sum of projections
\begin{equation}
\label{3.1.1} P(\Delta) = \sum_{k\in \Delta} P_k
\end{equation}
converges unconditionally, so the projections $P(\Delta) $ are well
defined and
\begin{equation}
\label{3.1.2} \sup_\Delta \|P(\Delta)\| \leq M(v) < \infty.
\end{equation}
Invariant subspaces $E(\Delta) = Ran \, P(\Delta) $ have $\{P_k, \; k
\in \Delta \}$ as their Riesz system of projections, $ \dim \, P_k =
2. $

Could $P_k $ be split to give a basis of root functions for
$E(\Delta)?$  We put the question in this way because for one and the
same operator $L_{Per^\pm}(v) $ the answer could be {\em yes } and
{\em no} depending on $\Delta. $  For example, if
$$
v(x) = a e^{-10ix} +be^{10ix}
$$
and
\begin{equation}
\label{3.3.1} \Delta_0 = \{n\in \Gamma^\pm: \; \; n \not \equiv 0
\mod 5\},
\end{equation}
then the answer is positive, but for $\Delta_1 = 5 \mathbb{N} $  the
answer is {\em no} if $|a| \neq |b|, $ and {\em yes} if $|a|=|b|.$ We
explain this phenomenon in Section~4 (see Proposition~\ref{prop20}).

In view of (\ref{3.1.1}) and (\ref{3.1.2}), the following holds (see
Corollary 10 in \cite[Section 3]{DM28} for details).
\begin{Remark}
\label{rem1}  If $\Delta $ is an infinite set of even (or odd)
integers, then the corresponding system of periodic (or antiperiodic)
root functions  contains a basis of $E(\Delta)$ if and only if it
contains an unconditional basis of $E(\Delta).$
\end{Remark}

The spectra localization formula (\ref{2.2}) allows us to apply the
Lyapunov--Schmidt projection method (see \cite[Lemma 21]{DM15}) and
reduce the eigenvalue equation $Ly = \lambda y $ to a series of
eigenvalue equations in two-dimensional eigenspaces $E_n^0$ of the
free operator. This leads to the following (see \cite[Section
2.2]{DM15}).

\begin{Lemma}
\label{lem1} Let $L$ be a Hill operator with a potential $v\in L^2.$
 Then, for large enough $n\in \mathbb{N},$
there are functionals $\alpha_n (v;z) $ and $ \beta^\pm_n (v;z), \;
|z| < n $  such that a number $\lambda = n^2 + z, \;|z| < n/4, $ is a
periodic (for even $n$) or anti-periodic (for odd $n$) eigenvalue of
$L$ if and only if $z$ is an eigenvalue of the matrix
\begin{equation}
\label{p1}  \left [
\begin{array}{cc} \alpha_n (v;z)  & \beta^-_n (v;z)
\\ \beta^+_n (v;z) &  \alpha_n (v;z) \end{array}
\right ].
\end{equation}
Moreover, $\alpha_n (z;v) $ and $\beta^\pm_n (z;v)$ depend
analytically on $v$ and $z,$  and $z_n^-=\lambda_n^- - n^2$ and
$z_n^+=\lambda_n^+ - n^2$ are the only solutions of the equation
\begin{equation}
\label{p2}   (z-\alpha_n (v;z))^2=  \beta^-_n (v;z)\beta^+_n (v;z).
\end{equation}
\end{Lemma}

The functionals $\alpha_n (v;z) $ and $\beta^\pm_n (v;z)$ are well
defined for large enough $n$ by explicit expressions in terms of the
Fourier coefficients of the potential (see \cite[Formulas
(2.16)-(2.33)]{DM15} for Hill operators with $L^2$-potentials).

Here we provide formulas for $\alpha_n (v;z) $ and $\beta^\pm_n
(v;z)$ using the combinatorial approach that has been developed in
\cite{DM11,DM10} and used there to obtain the asymptotics of the
spectral gaps $\gamma_n = \lambda^+_n -\lambda_n^- $ for potentials
of the form $v(x) = a \cos2x + b \cos 4x.$

 For each $n\in \mathbb{N}$  a {\em walk} $x$
from $-n$ to $n$ (or from $n$ to $-n$ or from $n$ to $n$) is defined
through its {\em sequence of steps}
\begin{equation}
  \label{2.21}
x=(x(t))_{t=1}^{\nu+1}, \quad 1\leq \nu=\nu(x)<\infty,
\end{equation}
where $x(t) \in 2 \mathbb{Z} \setminus \{0\},$ and respectively,
\begin{equation}
  \label{2.22}
\sum_{t=1}^{\nu+1} x(t) = 2n   \quad  \left ( \text{or}  \quad
\sum_{t=1}^{\nu+1} x(t) = -2n \quad \text{or}  \quad
\sum_{t=1}^{\nu+1} x(t) = 0 \right ).
\end{equation} A walk $x$ is called {\em
admissible} if its {\em  vertices} $j(t) = j(t,x)$ given,
respectively,  by
\begin{equation}
  \label{2.23}
j(0) = -n  \quad \text{or} \;\;j(0) = +n
\end{equation}
and
\begin{equation}
  \label{2.24}
j(t) =-n + \sum_{i=1}^t x(i) \quad   \text{or}\quad j(t) =
 n + \sum_{i=1}^t x(i), \quad 1\leq t \leq \nu+1,
\end{equation}
satisfy
\begin{equation}
  \label{2.25}
 j(t) \neq \pm n \quad \text{for} \;\; 1\leq t \leq \nu.
\end{equation}

Let
\begin{equation}
\label{22.54}  v= \sum_{m \in 2\mathbb{Z}} V(m) e^{imx}
\end{equation}
be the Fourier expansion of the potential $v$ with respect to the
system $\{e^{imx}, \; m\in 2\mathbb{Z}\},$ and let $X_n, Y_n $ and
$W_n$ be, respectively, the set of all admissible walks from $-n$ to
$n,$  from $n$ to $-n$ and from $n$ to $n.$ For each admissible walk
$x$ we set
\begin{equation}
  \label{2.26}
h_1(x;z) = \prod_{t=1}^\nu [n^2 - j(t)^2 +z]^{-1}, \quad h(x)=
h_1(x) \prod_{t=1}^{\nu+1} V(x(t));
\end{equation}
then
\begin{equation}
  \label{2.27}
\alpha_n (z) = \sum_{x\in W_n}  h(x,z), \;\; \beta_n^+ (z) =
\sum_{x\in X_n} h(x,z), \;\; \beta_n^- (z) = \sum_{x\in Y_n} h(x,z).
\end{equation}

The core of our approach is analysis of asymptotic behavior of the
functionals $\beta_n^\pm (z) =\beta_n^\pm (v;z).$  In particular,
the following criterion (which is a slight modification of Theorem 1
in \cite{DM25} or Theorem~2 in \cite{DM25a}) gives a constructive
approach to determine the basisness properties of the root function
system.

\begin{Criterion}
\label{crit1}
 Let  $v \in L^2 ([0,\pi]), $  and let
 $\Delta \subset \Gamma^+ $  (or
 $\Delta \subset \Gamma^- $) be an infinite set of sufficiently large
 numbers.
If $\Delta= \Delta_0 \cup \Delta_1, $ where
\begin{equation}
\label{a0} \beta_n^+ (z) \equiv  \beta_n^- (z) \equiv 0 \quad
\text{for} \;\; n \in \Delta_0,
\end{equation}
\begin{equation}
\label{a1} \beta_n^+ (0) \neq 0, \quad \beta_n^- (0)\neq 0 \quad
\text{for} \;\; n \in \Delta_1
\end{equation}
and there is a constant $c>0 $ such that
\begin{equation}
\label{a2}  c^{-1}|\beta_n^\pm (0)| \leq |\beta_n^\pm (z)| \leq c \,
|\beta_n^\pm (0)|, \quad \text{for} \;\; n \in \Delta_1, \;\; |z|
\leq 1,
\end{equation}
then:

(a)  for large enough  $n \in \Delta, $ the operator $L_{Per^\pm}(v)$
has in the disc $D_n =\{z: |z- n^2|<1 \} $ exactly one periodic (or
antiperiodic) eigenvalue of geometric multiplicity 2 if $n \in
\Delta_0, $ and exactly two simple periodic (or antiperiodic)
eigenvalues  if $n \in \Delta_1; $

(b)  the system of root functions of $L_{Per^\pm} (v) $ contains a
Riesz basis of $E(\Delta)$ if and only if
\begin{equation}
\label{a3} \limsup_{n\in \Delta_1} t_n (0) <\infty,
\end{equation}
where
\begin{equation}
\label{a4} t_n (z)=\max \{|\beta_n^-(z)|/|\beta_n^+(z)|,
|\beta_n^+(z)|/|\beta_n^-(z)|\}.
\end{equation}
\end{Criterion}
In the framework of this criterion one can explain practically all
known cases of existence or non-existence of bases consisting of root
functions of the operators $L_{Per^\pm} (v)$ for specific classes of
potentials $v$. For example, the main result in \cite{ShV09} follows
from Criterion~\ref{crit1}.

In general form, i.e., without the restrictions (\ref{a0}) -
(\ref{a2}), Criterion~\ref{crit1} is given in \cite{DM26-2} in the
context of 1D Dirac operators but the formulation and proof are the
same in the case of Schr\"odinger operators (see Proposition~19 in
\cite{DM28}). Moreover, the same argument gives the following more
general statement.

\begin{Criterion}
\label{crit2} Let $\Gamma^+ =2\mathbb{N}, $ $\Gamma^-=2\mathbb{N}-1 $
in the case of Hill operators with $H^{-1}_{per}$-potentials, and
$\Gamma^+ =2\mathbb{Z}, $ $\Gamma^-=2\mathbb{Z}-1 $ in the case of
one dimensional Dirac operators with $L^2$-potentials. There exists
$N_* = N_* (v)$ such that for $|n|>N_*$ the operator
$L=L_{Per^\pm}(v)$ has in the disc $D_n =\{z: |z- n^2|<n/2 \}$
(respectively $D_n =\{z: |z- n|<1/2 \}$) exactly two periodic (for
$n\in \Gamma^+$) or antiperiodic (for $n\in \Gamma^-$) eigenvalues,
counted with multiplicity. Let
    $$\mathcal{M}^\pm =\{n\in \Gamma^\pm: \; n
\geq N_*, \; \lambda^-_n \neq \lambda^+_n \}.$$

(a) If $\Delta \subset \Gamma^\pm $ is an infinite set such that $|n|
> N_*$ for $n \in \Delta, $
then the system of periodic (or antiperiodic) root functions contains
a Riesz basis in $E(\Delta)$ if and only if
\begin{equation}
\label{cr11} \limsup_{n\in \Delta \cap \mathcal{M}^\pm} t_n (z_n^*)
< \infty,
\end{equation}
where $z_n^* = \frac{1}{2} (\lambda^-_n + \lambda^+_n) -\lambda^0_n
$ with $\lambda^0_n = n^2 $ for Hill operators and $\lambda^0_n = n
$ for Dirac operators.

(b) The system of root functions of $L_{Per^\pm}(v)$ contains a Riesz
basis, (respectively, in $L^2 ([0,\pi])$ in the Hill case or in $L^2
([0,\pi],\mathbb{C}^2)$ in the Dirac case) if and only if
(\ref{cr11}) holds for $\Delta =\Gamma^\pm.$
\end{Criterion}

Another interesting abstract criterion of basisness is the following.
\begin{Criterion}
\label{crit3}

The system  of  root functions of the operator $L_{Per^\pm} (v)$
contains a Riesz basis in $E(\Delta)$ if only if
\begin{equation}
\label{a6} \limsup_{n\in \Delta \cap \mathcal{M}^\pm} \frac{|
\lambda_n^+ -\mu_n|}{|\lambda_n^+ - \lambda_n^-|} <\infty,
\end{equation}
where (for large enough $n$) $\mu_n $ is the Dirichlet eigenvalue
close to $n^2.$
\end{Criterion}

In the case $\Delta = \Gamma^\pm$ this criterion was given (with
completely different proofs) in \cite{GT12} for Hill operators with
$L^2$-potentials and in \cite{DM28} for Hill operators with
$H^{-1}_{per}$-potentials and for one-dimensional Dirac operators
with $L^2$-potentials as well. The proof of the criterion in the
more general case $\Delta \subset \Gamma^\pm $ is the same.

However, if one wants to apply Criterion~\ref{crit3} to specific
potentials $v,$ say $v(x) = a \cos 2x + b \cos 4x $ with $a,b \in
\mathbb{C},$ it is necessary first  to obtain the asymptotics of the
spectral gaps $|\lambda_n^+ - \lambda_n^-|$ and deviations $|\mu_n -
\lambda_n^+|,$  what is by itself quite a difficult problem.

In \cite{DM25a, DM25} we considered low degree trigonometric
polynomials with nonzero coefficients $v(x) $ of the form

 (i)  $ ae^{-2ix} +be^{2ix}; $

(ii)  $ ae^{-2ix} +Be^{4ix};  $

(iii)  $ ae^{-2ix} +Ae^{-4ix}  + be^{2ix} +Be^{4ix}. $

It is shown that the system of eigenfunctions and (at most finitely many)
associated functions is complete but it is not a basis in $L^2
([0,\pi], \mathbb{C})$ if  $|a| \neq |b| $  in the case (i),  if $|A|
\neq |B| $ and neither $-b^2/4B$ nor $-a^2/4A$ is an integer square
 in the case (iii), and it is never a basis  in the case
(ii) subject to periodic boundary conditions. In connection with
Example (iii) see also \cite{BSWZ12, Ro12}.

In this paper we extend the analysis of the above example (ii) to
potentials of the form
$$
v(x) = a e^{-2irx} + b e^{2isx}, \quad a, b \neq 0, \; r,s \in
\mathbb{N}, \; r\neq s.
$$
In Section~2, Theorem~\ref{thm3.1}, it is shown that the system of
root functions does not contain a basis in $L^2 ([0,\pi],
\mathbb{C}) $ for periodic $bc$ or if $bc$ is antiperiodic but $r,s
$ are odd.

In Section 3, the case $ r=1, \;s>2 $ any (i.e., odd or even) with
antiperiodic boundary conditions is completely analyzed as well, and
it is shown that the system of root functions does not contain a
basis in $L^2 ([0,\pi], \mathbb{C}) $  -- see Theorem~\ref{thm5}.

In our proofs we face series of questions related to enumerative
combinatorics and diophantine equations. Their solution would
dramatically extend the class of trigonometric polynomial potentials
$v(x) $ for which the problem of convergence of spectral
decompositions could be resolved. In our study of potentials (iii) in
\cite{DM11, DM10} we discover a combinatorial identity (see also
\cite{BSWZ12, Ro12}) that could be a prototype of such results. In
this connection see \cite{DM30} for more comments and open problems.

\bigskip

{\em Acknowledgement. } Some of the main results of this paper have
been obtained at the Mathematisches Forschungsinstitut Oberwolfach
during our three week stay there in August 2010 within the Research
in Pairs Programme. We appreciate the hospitality and creative
atmosphere of the Institute.

\section{Two exponential term potentials}

1. Our main objects are the potentials of the form
\begin{equation}
\label{3.3.2} v(x) =a e^{-2Rix} +b
e^{2Six}, \quad a, b \neq 0,
\end{equation}
with $R, S \in \mathbb{N}, \; R \neq S. $
Then
\begin{equation}
\label{3.3.5} R=dr, \quad S=ds, \quad\text{where} \; \; r, s \;\;
\text{are coprime};
\end{equation}
they are the main parameters in what follows.

In view of (\ref{3.3.2}), an admissible path $x=(x(t))_{t=1}^{\nu+1}$
from $-n $ to $n$ gives a non-zero term $h(x,z) $ in $\beta_n^+ (z) $
(see (\ref{2.27})) if and only if
\begin{equation}
\label{3.3.6} x(t) \in \{-2R , 2S \}, \quad t=1, 2, \ldots, \nu+1.
\end{equation}
Let $x$ be such
a path, and let
\begin{equation}
\label{3.4.3} \tilde{p} = \# \{t: \;  x(t) = -2R\}, \quad \tilde{q}
= \# \{t: \;  x(t) = 2S\}.
\end{equation}
Consider
\begin{equation}
\label{3.4.4} n \in \Delta := (rsd) \mathbb{N}, \quad \text{i.e.,}
\;\; n= rsdm, \;\; m \in \mathbb{N};
\end{equation}
then
\begin{equation}
\label{3.4.5}
-2 R \tilde{p} + 2S \tilde{q} = 2n, \quad s \tilde{q} = r \tilde{p} + rsm.
\end{equation}
and therefore,
\begin{equation}
\label{3.4.6}
\tilde{q} = r q, \quad \tilde{p} = sp \quad \text{with} \;\; q= p+m.
\end{equation}
Under the assumptions (\ref{3.4.3}) - (\ref{3.4.6}) we denote by
$X_n(p) $ the set of all admissible paths from $-n$ to $n$ with
$\tilde{p}= ps $ negative steps $-2R$ and $\tilde{q}= qr $ positive
steps $2S.$ Then $n \in \Delta $  (see (\ref{3.4.4}) ) implies
\begin{equation}
\label{3.5.1} \# X_n(0) = 1, \quad X_n(0) = \{ x^* \},
\end{equation}
where
\begin{equation}
\label{3.5.2} x^* (k)= 2sd, \quad    j^*_k  := j(k,x^*) = -n+ 2sdk,
\quad 1 \leq k \leq rm-1.
\end{equation}
Therefore, for $n= rsdm$ we have $n^2 - (j^*_k)^2= 4s^2d^2k (rm-k), $
which implies that
\begin{equation}
\label{3.5.20} h(x^*,0)=b^{rm} \prod_{k=1}^{rm-1} \frac{1}{n^2 -
(j^*_k)^2}= \frac{b^{rm}}{(4s^2d^2)^{rm-1} [(rm-1)!]^2}.
\end{equation}

Moreover, in these notations, we have
\begin{equation}
\label{3.5.3} \beta_n^+ (z) = \sum_{p=0}^\infty \sum_{x\in X_n(p)}
h(x,z),
\end{equation}
where, for $x \in X_n(p),$
\begin{equation}
\label{3.5.4}
 h(x,z)= a^{\tilde{p}} b^{\tilde{q}} h_1 (x,z), \quad
 h_1 (x,z) =
  \prod_{t=1}^{\tilde{p} + \tilde{q} -1}
 \left ( n^2 - j(t,x)^2 +z  \right )^{-1}.
\end{equation}

2. Next we show that the leading term in the asymptotics of
$\beta_n^+ (z) $ is determined by $h(x^*, 0) $ only. Fix $p \geq 1 $
and $x\in X_n(p); $  choose a set of vertices $j(t_k, x), \; k=1,
\ldots rm-1 $ so that
\begin{equation}
\label{3.5.5}
0 \leq \delta_k := j^*_k - j(t_k, x) < 2S = 2sd.
\end{equation}
(This is possible since the positive steps of $x$ are equal to $2S.$)

We have $h_1 (x,z) = \Pi_1 (z) \cdot \Pi_2 (z),  $ where
$$
\Pi_1 (z)= \prod_{k=1}^{rm-1} (n^2 - j(t_k,x)^2+z)^{-1}
$$
and $\Pi_2 (z) $ is the product of those factors of $h_1 (x,z) $
which are not included in $\Pi_1 (z). $ In view of (\ref{3.4.3}) and
(\ref{3.4.6}), the number of factors in $\Pi_2 (z) $ is equal to
$$\nu (x) - (rm-1) = \tilde{p} +\tilde{q}-1  - (rm-1) = (r+s)p.  $$
For $n\geq 2 $ and $|z|\leq 1 $ we have $$|n^2 - j(t_k,x)^2+z| \geq
|n^2 - j(t_k,x)^2| -1 \geq 2n-2 \geq n, $$ so the absolute value of
each factor is less than $1/n.  $ Therefore,
\begin{equation}
\label{3.6.0} |\Pi_2 (z)| \leq (1/n)^{(r+s)p}.
\end{equation}

To estimate $\Pi_1 (z) $ we need the following (compare with
\cite[Lemma~2]{DM25}).

\begin{Lemma}
 \label{lem3.0}
 If $\{j_1, \ldots, j_K\} \subset \{j= -n+2t, \; t=1, \ldots,
n-1\},$ then for large enough $n$
 and $|z|\leq 1 $
 \begin{equation}
  \label{3.380}
\prod_{k=1}^K |n^2-j_k^2 +z|^{-1} = \left (\prod_{k=1}^K
|n^2-j_k^2|^{-1} \right ) (1+ \theta_n),   \quad
 |\theta_n | \leq \frac{4\log n}{n}.
  \end{equation}
  \end{Lemma}

\begin{proof} Indeed, we have
$$ \theta_n =   \prod_{k=1}^K \frac{n^2-(j_k)^2}{n^2-(j_k)^2 +z}
-1 = e^{-w_n} -1, $$ where $w_n = \sum_{k=1}^K \log \left (1+
\frac{z}{n^2 -(j_k)^2} \right ).$  Therefore,
 by the inequality $$ |\log
(1+\zeta)| \leq \sum_{k=1}^\infty |\zeta|^k \leq 2 |\zeta|  \quad
\text{for} \;\; |\zeta|\leq 1/2, $$  it follows that for large enough
$n$ $$ |w_n| \leq \sum_{k=1}^K \frac{2|z|}{n^2 -(j_k)^2} \leq
\sum_{k=1}^{n-1} \frac{2}{n^2 - (-n+2k)^2 } = \frac{1}{n}
 \sum_{k=1}^{n-1} \frac{1}{k} \leq \frac{2\log n}{n}< \frac{1}{2}. $$
On the other hand, if $|w|\leq 1/2$ then $|e^{-w} -1| \leq
\sum_{k=1}^\infty |w|^k \leq 2|w|,$ which implies  (\ref{3.380}).
\end{proof}

Now we could estimate the product $\Pi_1 (z)$ by Lemma~\ref{lem3.0}.
Indeed, if $j_k =j(t_k, x) $ then due to the choice of $t_k$ (see
(\ref{3.5.5})) the vertices $j_k $ are distinct and $ -n <j_k <n. $
Therefore, (\ref{3.380}) implies that
 \begin{equation}
  \label{3.38}
 \Pi_1 (z)  =  \Pi_1 (0) (1+ \theta_n),   \quad \text{where}\;\;
 |\theta_n| =O \left (  \frac{ \log n}{n} \right ).
  \end{equation}

 3. Next we estimate $\Pi_1 (0) $ by comparing it with $h_1
(x^*,0).$ To this end we need the following.
\begin{Lemma}
\label{lem3.1} Let $n, K, S \in \mathbb{N} $ and $n \geq (K+1) S, $
and let
\begin{equation}
\label{3.6.1}
j_k = \pm (n - 2kS), \quad  0 \leq  \delta_k \leq 2(S-d),
\quad k = 1, \ldots, K, \quad  d \in (0, S).
\end{equation}
Then
\begin{equation}
\label{3.6.2}
\prod_{k=1}^K \frac{n^2 - (j_k)^2 }{n^2 -(j_k - \delta_k)^2}
\leq C n^{1- d/S}.
\end{equation}
\end{Lemma}
(This lemma is a more general assertion than Lemma 12 in \cite{DM25},
where $S=2$ and $\delta_k=2$ so $d=1.$)

\begin{proof}
First we consider the case $j_k = -(n-2kS),$ i.e., moving forward
from $-n$ to $+n.$  Then $ n-j_k = 2n - 2kS  \geq 2S, $ and we have
$$
\frac{n^2 - (j_k)^2 }{n^2 -(j_k - \delta_k)^2}
=\frac{(n + j_k)(n - j_k) }{(n +j_k - \delta_k)(n -j_k + \delta_k)}
\leq \frac{n + j_k}{n + j_k - \delta_k}.
$$

If $j_k = -n + 2Sk, $ then
$$
\frac{n + j_k}{n + j_k - \delta_k}
= \left ( 1- \frac{\delta_k}{2kS}  \right )^{-1} \leq
\left ( 1- \frac{S-d}{kS}  \right )^{-1}.
$$
Therefore, the product in (\ref{3.6.2}) does not exceed
$$
\prod_{k=1}^K \left ( 1- \frac{\gamma}{k}  \right )^{-1} \leq C n^\gamma
\quad \text{where} \;\; \gamma = 1- \frac{d}{S}, \;\; C= C(\gamma).
$$

When we are moving backward from $+n$ to $-n,$ then  $j_k = n - 2Sk,
$ so
$$
\frac{n + j_k}{n + j_k - \delta_k}
= \left ( 1- \frac{\delta_k}{2n - 2kS}  \right )^{-1} \leq
\left ( 1- \frac{(S-d)}{(K+1-k)S}  \right )^{-1}.
$$
Therefore, the product in (\ref{3.6.2}) does not exceed
$$
\prod_{k=1}^K \left ( 1- \frac{\gamma}{K+1-k}  \right )^{-1} \leq C n^\gamma
\quad \text{where} \;\; \gamma = 1- \frac{d}{S}, \;\; C= C(\gamma),
$$
which completes the proof.

\end{proof}

4. By Lemma \ref{lem3.0}, $|\Pi_1 (z)/\Pi_1 (0)| = 1+ O \left ((\log
n)/n \right ).$ On the other hand,  applying Lemma~\ref{lem3.1} to
$\Pi_1 (0)/h_1 (x^*,0) $ we obtain (since $S=sd$)
$$
\Pi_1 (0) \leq C n^{1-1/s} h_1 (x^*,0).
$$
Together with the estimate (\ref{3.6.0}) for $\Pi_2, $
this leads to
$$
|h_1 (x,z) |/h_1 (x^*,0) \leq  C n^{1-1/s} (1/n)^{(s+r)p}.
$$

Let us take into account the coefficients $a, b $ of the potential.
We set
\begin{equation}
\label{3.8.4} T= \max \{  |a|, |b| \};
\end{equation}
then
\begin{equation}
\label{3.9.1} \frac{|h(x,z)|}{|h(x^*,0)|} = \frac{|a|^{\tilde{p}}
|b|^{\tilde{q}} |h_1 (x,z)|} {|b|^{rm} h_1 (x^*,0)} \leq C n^{1-1/s}
\left (\frac{T}{n} \right )^{(s+r)p},
\end{equation}
because $\tilde{p} + \tilde{q} - rm = (r+s) p $ due to (\ref{3.4.3})
and (\ref{3.4.6}). \bigskip

5. The number of  paths $x \in X_n(p) $ does not exceed
\begin{equation}
\label{3.10.1} \# X_n(p) \leq \begin{pmatrix} \tilde{p} + \tilde{q}\\
\tilde{p}  \end{pmatrix}.
\end{equation}
In view of (\ref{3.4.6}),
\begin{equation}
\label{3.10.2} \# X_n(p) \leq \begin{pmatrix} (s+r)p + rm \\ sp
\end{pmatrix} \leq
\begin{cases}
\frac{1}{(sp)!}  [(s+2r)m]^{sp}   \quad  & \text{if} \;\; p\leq m,\\
2^{(s+2r)p}       \quad  & \text{if} \;\; p>m.
\end{cases}
\end{equation}
By (\ref{3.9.1}) and (\ref{3.10.2}), it follows that
\begin{equation}
\label{3.11.1} \sum_{p=1}^\infty \sum_{x \in X_n(p) } |h(x,z)| \leq
|h(x^*,0)| ( \sigma_1 + \sigma_2),
\end{equation}
where
$$
\sigma_1 = C n^{1-\frac{1}{s}} \sum_{p=1}^m  \frac{1}{(sp)!}
[(s+2r)m]^{sp}
 \left (\frac{T}{n} \right )^{(s+r)p},
$$
$$
\sigma_2 = C n^{1-\frac{1}{s}}  \sum_{p=m+1}^\infty  2^{(s+2r)p}
 \left (\frac{T}{n} \right )^{(s+r)p}.
$$
Since $ n= rsdm $ we have
$$
[(s+2r)m]^{sp} \left (\frac{T}{n} \right )^{(s+r)p} =
\left ( T\frac{s+2r}{rsd}   \right )^{sp} \left (\frac{T}{n} \right )^{rp}
=\left (\frac{T_1}{n} \right )^{rp}
$$
where $T_1 = T \left ( T \frac{s+2r}{rsd}   \right )^{s/r}.$
Therefore, for $n \geq 2T_1 + 1, $
$$
\sigma_1 = C n^{1-\frac{1}{s}}  \sum_{p=1}^m  (T_1/n)^{rp} \leq 2 C
n^{1-\frac{1}{s}}  (T_1/n)^{r} \leq C_1 n^{1- r -\frac{1}{s} },
$$
where $C_1=  C_1 (r,s, T). $

The second sum $\sigma_2 $ is much smaller than the first one:
$$
\sigma_2 \leq C n^{1-1/s} \sum_{p=m+1}^\infty \left (\frac{4T}{n}
\right )^{(s+r)p} \leq C_2 n^{1-(s+r)(m+1)-\frac{1}{s}},
$$
where $C_2=  C_2 (r,s, T). $

In view of (\ref{3.11.1}), the obtained estimates for $\sigma_1 $
and $\sigma_2 $ prove that
\begin{equation}
\label{3.13.2} \sum_{p=1}^\infty \sum_{x \in X(p) } |h(x,z)| \leq C
(r,s,T) |h(x^*,0)| \, n^{1- r -\frac{1}{s} }, \quad |z|\leq 1.
\end{equation}
Hence, the following is true.
\begin{Lemma}
\label{lem3.2} For large enough $n= mdsr, \; m \in \mathbb{N},$
\begin{equation}
\label{3.13.31} \frac{1}{2} \beta^+_n (0) \leq |\beta^+_n (z)| \leq
 2\beta^+_n (0)
\end{equation}
and
\begin{equation}
\label{3.13.3} \beta_n^+ (0) =
 h(x^*,0) \left ( 1+ O  \left( n^{1- r -\frac{1}{s}}  \right  )  \right )
\end{equation}
with
\begin{equation}
\label{3.13.4} h(x^*,0) = 4s^2d^2 \left ( \frac{b}{4s^2d^2} \right
)^{rm} ((rm-1)!)^{-2}.
\end{equation}

\end{Lemma}
\bigskip

6. To analyze the paths $y \in Y_n $ from $n$ to $-n, $ i.e.,
\begin{equation}
\label{3.14.1}
\sum_1^{\nu + 1} y(t) = -2n,
\end{equation}
we can just exchange the roles of $R$ and $S$ and repeat the above
statements with proper adjustments. Then
$$
Y_n(0) = \{y^* \}, \quad y^* (t) = - 2R, \;\; 1 \leq t \leq sm-1,
$$
and the following holds.
\begin{Lemma}
\label{lem3.3} For large enough $n= mdsr, \; m \in \mathbb{N},$
\begin{equation}
\label{3.14.31} \frac{1}{2} \beta^-_n (0) \leq |\beta^-_n (z)| \leq
 2\beta^-_n (0)
\end{equation}
and
\begin{equation}
\label{3.14.3} \beta_n^- (0) =
 h(y^*,0) \left ( 1+ O  \left( n^{1- s -\frac{1}{r}}  \right  )  \right )
\end{equation}
with
\begin{equation}
\label{3.14.4} h(y^*,0) = 4r^2d^2 \left ( \frac{a}{4r^2d^2} \right
)^{sm} ((sm-1)!)^{-2}.
\end{equation}
\end{Lemma}

\begin{Remark}
If $R=1 $ then $d=r=1, \; S=s,$   and for any $n$ if we go backward
from $+n$ to $-n$ it could be done without using forward steps $
+2s.$ Analogues of (\ref{3.14.3}) could be given for any $s$ -- see
Section~3.5.
\end{Remark}
\bigskip

7.  The set $ \Delta $ defined in (\ref{3.4.4}) certainly contains
infinitely many {\em even}  integers because $ m $ could run over $ 2
\mathbb{N}.$ But if $rsd $ is even, then $\Delta \cap (2 \mathbb{N}
+1) = \emptyset $ while  $ \Delta \cap (2 \mathbb{N} +1) $ is
infinite if $rsd$ is odd, i.e., if $R$ and  $S$ are odd. In any case,
if $R\neq S, $ say $R<S, $
$$
\min \{|\beta_n^\pm (0)/\beta_n^\mp (0)|, \; n = (rsd)m \} \leq
(C_3)^m   \frac{(rm-1)!}{(sm-1)!} \leq (C_4)^m m^{-|r-s|m}.
$$
In view of Criterion~\ref{crit1}, these observations lead to the
following.
\begin{Theorem}
\label{thm3.1} For any potential $v$ in (\ref{3.3.2}) there is {\em
no} basis consisting of root functions of $L_{Per^+} (v). $ If $R$
and $S$ are odd, the same is true for $L_{Per^-} (v). $
\end{Theorem}

\section{Potentials $ a e^{-2ix} +b e^{2six}, \; s>2.$}

1. If we analyze $ bc = Per^- $ in the case the potential is of the
form (\ref{3.3.2}) and one of the parameters $r, s $ in (\ref{3.3.5})
is even then the constructions in Section~2 cannot be applied to give
us a negative statement like Theorem~\ref{thm3.1}. In this section we
present elaborate analysis in the case $r=1, \, s>2 $ and
\begin{equation}
\label{4.1.1}  \Delta = \{n= sm-1, \; m \in \mathbb{N}\}.
\end{equation}
Observe, that if $s $ is even, then $\Delta$ consist of odd numbers,
and if $s$ is odd then $\Delta \cap (2\mathbb{N}-1) \neq \emptyset $
and $\Delta \cap 2\mathbb{N} \neq \emptyset. $ So, by showing that
$$
\inf \{ |\beta_n^\pm (0)/ |\beta_n^\mp (0)| : \quad n \in \Delta,
\;\; n \geq N(v)\} = 0
$$
we would obtain by Criterion~\ref{crit1} that there is {\em no} basis
in $L^2([0,\pi])$ consisting of root functions of $L_{Per^-} (v) $
for potentials of the form
\begin{equation}
\label{4.1.4} v(x) = a e^{-2ix} +B e^{2six}, \quad a, b \neq 0,
\;\; s \geq 3.
\end{equation}
Let us remind that Theorem~\ref{thm3.1} in Section~2 considers the
operators $L_{Per^+} (v) $ for any $s.$ Its claim follows from
Criterion~\ref{crit1} because
\begin{equation*}
\label{4.2.1} \inf \{ |\beta_n^\pm (0)/ |\beta_n^\mp (0)| : \quad n
\in s \mathbb{N}, \; n \geq N(v)\} = 0.
\end{equation*}
In the sequel we write for convenience $h_1(x) $  instead of
$h_1(x,0),$ and $h(x) $  instead of $h(x,0).$

\bigskip

2. Fix $n = sm-1; $ a path $x= (x(t))_{t=1}^{\nu+1} $ from $-n$ to
$n$ gives a non-zero term $h(x,z) $ in $\beta^+_n (z)$ if and only
if (compare with (\ref{3.3.6}))
\begin{equation} \label{4.2.2} x(t) = -2 \quad \text{or} \quad x(t)= 2s.
\end{equation}
Set
\begin{eqnarray}
\label{4.2.3}
p= \#\{t: \; x(t) = -2,  \quad 1 \leq t \leq \nu (x) +1\}, \\
\nonumber q= \#\{t: \; x(t) = 2s,  \quad 1 \leq t \leq \nu (x) +1\};
\end{eqnarray}
then we have
\begin{equation}
\label{4.2.5} 2n = -2p+ 2sq \Rightarrow sm-1 = -p + sq \Rightarrow  p
= 1 + s(q-m).
\end{equation}
 We set
\begin{equation}
\label{4.3.1} p= 1+ s \kappa, \quad  q = m+ \kappa
\end{equation}
to satisfy (\ref{4.2.5}), and define $X_n (\kappa) $ as the set of
all admissible paths satisfying (\ref{4.2.2}) which parameters $p $
and $q $ are given by (\ref{4.3.1}). Then
\begin{equation}
\label{4.3.2} \# X_n (0) = m+1,
\end{equation}
and with $p= 1, \; q= m $ a path $ \xi^\tau \in X_n (0) $ is uniquely
determined by the position $\tau $ of its only step $-2.$ In other
words, the paths in $X_n (0)$ are given by
\begin{equation}
\label{4.3.3}
\xi^\tau (t) = \begin{cases}
2s,  \quad  t \neq \tau
\\  -2,   \quad  t = \tau
\end{cases}  \qquad    1 \leq \tau, \,t \leq m+1.
\end{equation}
Among them the two paths $\xi^1 $ and $\xi^{m+1} $ are special in
the sense that $ h_1 (\xi^1) = h_1 (\xi^{m+1}) < 0,  $ while $ h_1
(\xi^\tau) > 0 $ for $ \tau = 2, \ldots , m. $ More precisely, since
$ j(t,\xi^1) = -n-2 + 2s(t-1), $ we have
$$
n^2 - j(1,\xi^1 )^2 = n^2 - (-n-2)^2 = - 4(n+1) = -4ms
$$
and
$$
n^2 - j(t+1,\xi^1)^2 = n^2 - (-n-2 + 2st)^2 =
4s(m-t)(st-1), \;\; t=1, \ldots, m-1,
$$
so it follows that
\begin{equation}
\label{4.4.4} h_1 (\xi^1) = \prod_{t=1}^m  [n^2 - j(t,\xi^1)^2]^{-1}
= \frac{-1}{(4s)^m m!} \left ( \prod_{t=1}^{m-1} (st-1)   \right
)^{-1}.
\end{equation}
By symmetry $ h_1 (\xi^{m+1} ) =  h_1 (\xi^1 ), $ so we obtain for
their sum
\begin{equation}
\label{4.4.5} h_1 (\xi^1 ) + h_1 (\xi^{m+1} ) = - H^-(m)
\end{equation}
with
\begin{equation}
\label{4.4.5a}
 H^- (m) = \frac{2}{(4s)^m m!} \left ( \prod_{t=1}^{m-1} (st-1)   \right
)^{-1} = \frac{2s}{(2s)^{2m} m!}\frac{\Gamma (1-\frac{1}{s})}{\Gamma
(m-\frac{1}{s})}.
\end{equation}

For $ \xi^\tau  $ with $ 2 \leq \tau \leq m $ we have
\begin{equation}
\label{4.5.3}
j(t,\xi^\tau)  =  \begin{cases}
-n + 2st, \quad t \leq  \tau-1, \\
-n-2 + 2s(t-1), \quad \tau \leq t \leq m.
\end{cases}
\end{equation}
By (\ref{4.1.1}) and (\ref{4.5.3})
$$
n^2 - j(\xi^\tau, t)^2 = \begin{cases}
4st [(m-t)s -1], \quad  1 \leq t \leq \tau -1,\\
4s(s(t-1)-1)(m-(t-1)), \quad \tau  \leq t \leq  m,
\end{cases}
$$
which implies, for $ 2\leq \tau \leq m, $ that
\begin{equation}
\label{4.5.5} h_1 (\xi^\tau ) = \frac{1}{(4s)^m (\tau -1)! (m-\tau
+1)! } \left ( \prod_{t=m-\tau +1}^{m-1} (st-1)   \right )^{-1} \left
( \prod_{t=\tau -1}^{m-1} (st-1)   \right )^{-1}.
\end{equation}
One can easily see that the sum
\begin{equation}
\label{4.6.0} H^+ =H^+ (m):= \sum_2^m h_1 (\xi^\tau )
\end{equation}
 can be written (if we change $\tau $ to $\tau -1 $) as
\begin{equation}
\label{4.6.1} H^+ =  \frac{1}{(4s)^m} \sum_{\tau=1}^{m-1}
\frac{\prod_1^{\tau-1} (st-1) }{ \tau!} \frac{\prod_1^{m-\tau-1}
(st-1) }{ (m-\tau)!} \left ( \prod_{t=1}^{m-1} (st-1)   \right
)^{-2}.
\end{equation}
We set $\alpha = 1/s; $ then
\begin{equation}
\label{4.6.10} \alpha < 1/2 \quad (\text{so} \;\; 1 -2\alpha > 0)
\quad \text{for} \;\; s>2.
\end{equation}
Let
\begin{equation}
\label{4.6.2} A_\alpha (k) = \frac{\alpha \prod_1^{k-1} (t-\alpha)
}{k!} = \frac{ \alpha \Gamma (k-\alpha )}{\Gamma (1-\alpha )\Gamma
(k+1 )},\quad k\geq 2,
\end{equation}
\begin{equation}
\label{4.6.3}
A_\alpha (0) = 0, \quad A_\alpha (1) =\alpha.
\end{equation}
Then
\begin{equation}
\label{4.6.4}
2 A_\alpha (m) \times (H^+/H^-) =
\sum_{\tau =1}^{m-1} A_\alpha (\tau) A_\alpha (m- \tau),
\end{equation}
and
\begin{equation}
\label{4.6.5} \sum_{k =0}^\infty  A_\alpha (k) w^k = f_\alpha (w):=1
- (1-w)^\alpha
\end{equation}
happens to be a nice generating function. The right-hand side of
(\ref{4.6.4}) is the $m$-th Taylor coefficient $T_m $ of the square
$$
(f_\alpha (w))^2=(1 - (1-w)^\alpha )^2 = 1 -2(1-w)^{\alpha} +
(1-w)^{2\alpha} = 2f_\alpha (w)- f_{2\alpha} (w),
$$
so it equals
$$
T_m ([f_\alpha]^2) = 2A_\alpha (m) - A_{2\alpha} (m)
$$
Hence, dividing by $2A_\alpha (m) $
and taking into account (\ref{4.6.10}) and  (\ref{4.6.2}),
we obtain
\begin{equation}
\label{4.7.2} \frac{H^+}{H^-} = 1 -  \frac{A_{2\alpha}(m)}{2A_\alpha
(m)}= 1- \frac{\Gamma (1-\alpha) \Gamma(m-2\alpha )} {\Gamma
(m-\alpha )\Gamma (1-2\alpha)}, \quad \alpha = 1/s.
\end{equation}
The Stirling formula shows that
\begin{equation}
\label{4.7.3} r(m) := \frac{A_{2\alpha}(m)}{2A_\alpha (m)}=
\frac{\Gamma (1-\alpha) \Gamma(m-2\alpha )} {\Gamma (m-\alpha
)\Gamma (1-2\alpha)} = \frac{\Gamma (1-\alpha)} {\Gamma (1-2\alpha)}
\rho (m)m^{-\alpha},
\end{equation}
where  $  \rho (m) \to 1. $ Therefore,
\begin{equation}
\label{4.7.4}
 \frac{H^- - H^+}{H^- + H^+} = \frac{r(m)}{2- r(m)}
 \asymp  \frac{1}{2}r(m)
\end{equation}
for large enough $m,$ i.e., we proved the following.
\begin{Lemma}
\label{lem4.1} In the above notations,
\begin{equation}
\label{4.8.4} H^- (m) - H^+ (m)  \gtrsim  m^{-1/s} \left (H^- (m) +
H^+ (m) \right ) \quad \text{as}\;\; m\to \infty.
\end{equation}
\end{Lemma}

By Lemma~\ref{lem3.0},  for large enough $n$ and $|z| \leq 1$ we
have that
\begin{equation}
\label{49.1} h_1 (\xi, z) = h_1 (\xi,0) (1+\theta(\xi,z)), \;\;
|\theta(\xi,z)| \leq \frac{4\log n}{n}, \;\; \xi \in X_n (0).
\end{equation}
Indeed, if $\xi= \xi^\tau, \; \tau = 2, \ldots, m-1 $    then
(\ref{49.1}) follows directly from Lemma~\ref{lem3.0}. To handle
$h_1 (\xi^1, z), $ we write it in the form $$h_1 (\xi^1, z)=
\frac{1}{n^2- (-n-2)^2 +z} \prod_{k=1}^{m-1}\frac{1}{n^2 -(-n-2
+2sk)^2 +z}. $$ Then we apply Lemma~\ref{lem3.0}) to the product
on the right and estimate the single factor by $(-4n-4+z)^{-1}=
-(4n+4)^{-1} (1+ O(1/n).$ The case $\xi= \xi^{m+1}$ is symmetric.

From  (\ref{49.1}) and (\ref{4.4.5}) it follows that
\begin{equation}
\label{49.2} h_1(\xi^1, z) +h_1(\xi^{m+1},z) =-H^-(m) \, \left [ 1+
O((\log n)/n) \right ].
\end{equation}
On the other hand, by (\ref{4.6.0}) and (\ref{49.1}) we obtain that
$$ \sum_{\tau=2}^m h_1 (\xi^\tau, z)=\sum_{\tau=2}^m h_1 (\xi^\tau)
(1+\theta(\xi^\tau,z)) = H^+(m) + \Omega,  $$ where
$$|\Omega| = \left |\sum_{\tau=2}^m h_1 (\xi^\tau)
\theta(\xi^\tau,z))\right | \leq \sum_{\tau=2}^m h_1
(\xi^\tau)\frac{4\log n}{n} =H^+ (m) \frac{4\log n}{n}. $$ Thus,
we have
\begin{equation}
\label{49.3} \sum_{\tau=2}^m h_1 (\xi^\tau, z)= H^+ (m) [1+ O
((\log n)/n)].
\end{equation}
Now (\ref{49.2}) and (\ref{49.3}) give us, for  $|z|\leq 1,$ that
\begin{equation}
\label{49.4} \sum_{\xi\in X_n (0)} h_1 (\xi^\tau, z)= (H^+ (m) -H^-
(m)) \, \left [1+ O ((\log n)/n) \right].
\end{equation}
\bigskip

3.  Next we estimate the ratio of $$ \sum_{x \in X_n(\kappa)} |h_1
(x,z)|\quad \text{and} \sum_{\xi \in X_n(0)} |h_1 (\xi) | = H^- +
H^+. $$ Fix  $x \in X_n (\kappa), \; \kappa \geq 1, $ and  set
\begin{equation}
\label{4.9.1} \tau = \min \left [m+1, \min\{t: \; x(t) = -2\} \right
].
\end{equation}
Let
\begin{equation}
\label{4.9.2} j^*_k = j(k,\xi^{\tau}), \quad k=1, \ldots m.
\end{equation}
denote the vertices of $\xi^{\tau}.$

Next we choose $m$ vertices $j_k = j(t_k, x) $ of $x $ so that $j_k$
is ''close'' to $j^*_k $ as follows. If $\tau = m$ or $\tau
 = m+1 $ we
set $j_k = j_k^*, \; k=1, \ldots, m.$ If $ \tau < m $ we set
\begin{equation}
\label{4.9.4} t_k = k \quad  \text{if}  \;\; 1 \leq k \leq \tau
\end{equation}
and
\begin{equation}
\label{4.10.1} t_k = \min \{t>\tau: \; j(t, x) > j^*_{k-1}\}, \quad
\tau +1 \leq k \leq m.
\end{equation}

Let $J(x): =(j(t, x))_{t=1}^{\nu(x)}  $ be the sequence of the
vertices of $x.$ The sequence $(j_k)_{k=1}^m = (j(t_k, x))_{k=1}^m $
is a subsequence of  $J(x);$ let
$$
I(x) = (i_1, \ldots, i_\rho), \; \rho = \nu (x) - m = (1+s)\kappa
$$
be its complementary subsequence in $J (x). $ Consider  the mapping
\begin{equation}
\label{4.10.10}
 \Phi_\kappa: X_n (\kappa)  \to X_n(0) \times \mathbb{Z}^{(1+s)\kappa}, \quad
\Phi_\kappa (x) = (\xi^{\tau}, I(x)).
 \end{equation}

\begin{Lemma}
\label{lem4.2}
The mapping $\Phi_\kappa $ is injective.
\end{Lemma}

\begin{proof}
The lemma will be proved if we show that given $\Phi_\kappa (x)
=(\xi^{\tau}, I(x)) $ we can restore in an unique way the path $x$
(or equivalently, the sequence of its vertices $J(x)). $

In view of the construction, if  $\tau = m $ or $\tau = m+1 $  then
$$
J(x) = (j_1^*, \ldots, j^*_m, i_1, \ldots, i_\rho), \quad \rho =\nu (x) -m.
$$

In the case $ \tau < m $ we have to find the vertices $j_k $ and
their places in $J(x). $
 By (\ref{4.9.4}),
$$j_k = j(k, x) = j^*_k, \quad   1\leq k \leq \tau. $$
Consider the first term $i_1 $ of the sequence $I(x). $ By
(\ref{4.10.1}),  there is an integer $\mu_1$ such that $ 0 \leq \mu_1
\leq m - \tau $ and
$$
i_1 - j^*_{\tau} = -2 + 2s \cdot \mu_1;
$$
then $j_k = j(k,x) = j^*_{\tau} + 2s (k-\tau ) \; $ for $
 \; \tau +1 \leq k \leq  k_1:= \tau + \mu_1. $
If  $ k_1 = m $ we have $j(t,x)= i_{t-m} $ for $m+1 \leq t \leq \nu
(x), $ so $J(x) $ is restored.

Otherwise, we set
$$
\tau_1 = \min \{ t: \; i_{t+1} - i_t \not \in \{-2, 2s \}, \; 1 \leq
t < \rho \}.
$$
From (\ref{4.10.1}) it follows that $i_1, \ldots, i_{\tau_1} $ are
successive vertices of $x, $  so we have
$$
j(t,x) = i_{t - k_1}, \quad   k_1 + 1 \leq  t \leq  \tau_1 + k_1.
$$
Moreover, there is $\mu_2 \in \mathbb{N} $ such that
 $$
i_{\tau_1+1} - i_{\tau_1} = -2 + 2s \cdot \mu_2,
$$
which implies
$$
j_k = i_{\tau_1} + 2s (k-k_1), \quad
k_1 + 1 \leq k \leq k_2: = k_1 + \mu_2,
$$
so
$$
j(t,x) = i_{\tau_1} + 2s (t-\tau_1-k_1),   \quad \tau_1 + k_1 + 1
\leq t \leq  \tau_1 + k_2.
$$

In the case  $ k_2 = m $ we have $j(t,x)= i_{t-m} $ for $m+\tau_1 +1
\leq t \leq \nu (x), $ so $J(x) $ is restored. Otherwise, we set
$$
\tau_2 = \min \{ t: \; i_{t+1} - i_t \not \in \{-2, 2s \}, \; \tau_1
+1 \leq t < \rho \}.
$$
and continue by induction.
\end{proof}
\bigskip

Fix $x \in X_n(\kappa), $  and let $(j_k)_{k=1}^m $ and $\Phi (x) =
(\xi^{\tau}, I(x)) $ be defined as above. Then
\begin{equation}
\label{14.4} h_1 (x,z) = \prod_{k=1}^m (n^2 - j_k^2+z)^{-1} \cdot
\prod_{i \in I(x)} (n^2 - i^2+z)^{-1},
\end{equation}
and by Lemma~\ref{lem3.0} we have $$\prod_{k=1}^m (n^2 -
j_k^2+z)^{-1}=\left ( \prod_{k=1}^m (n^2 - j_k^2)^{-1} \right ) \left
(1+ O((\log n)/n)\right ).$$ On the other hand, by (\ref{4.9.4}) and
(\ref{4.10.1}), $j_k =j^*_k $ for $1 \leq k \leq \tau$ and $j^*_{k-1}
< j_k \leq j^*_k $ for $ \tau <k \leq m. $ Therefore, by
Lemma~\ref{lem3.1} we obtain $$ \frac{1}{h_1 (\xi^{\tau})}
\prod_{k=1}^m (n^2 - j_k^2)^{-1} = \prod_{k=\tau}^m \frac{n^2 -
(j^*_k)^2}{n^2 - (j_k)^2} \leq C n^{1-\frac{1}{s}}.$$ (Since $j^*_k =
n- 2s(m+1-k), $  we apply Lemma~\ref{lem3.1} after changing  the
summation index by $ \tilde{k} = m+1 -k$.) Thus, the above
inequalities imply that
\begin{equation}
\label{14.5} \prod_{k=1}^m (n^2 - j_k^2+z)^{-1} \leq C \,h(\xi^\tau)
\, n^{1-\frac{1}{s}}.
\end{equation}

 Let $X_n(\kappa, \tau)$ be the set of all $x \in X_n(\kappa) $ such that
(\ref{4.9.1}) holds. The sets $X_n(\kappa, \tau), \; 1\leq \tau \leq
m+1$ are disjoint,  and
$$ X_n(\kappa) = \bigcup_{\tau =1}^{m+1} X_n(\kappa, \tau). $$
In view of (\ref{14.4}) and (\ref{14.5}) we have
\begin{equation}
\label{4.14.5} \sum_{x \in X_n(\kappa, \tau)} |h_1 (x,z)| \leq  C
n^{1-\frac{1}{s}} |h_1 (\xi^\tau)|  \sum_{x \in X_n(\kappa, \tau)}
\prod_{i \in I(x)} |n^2 - i^2+z|^{-1}.
\end{equation}
By Lemma \ref{lem4.1} the mapping $\Phi_\kappa $ is injective, so the
sequence $I(x)= (i_1, \ldots, i_\rho) $ is uniquely determined by $x
\in X_n (\kappa, \tau).$ Moreover,  $$ |n^2 - i^2 +z| \geq |n^2- i^2|
- 1 \geq \frac{1}{2} |n^2 - i^2| \quad \text{for} \;\; |z|\leq 1. $$
Therefore,
$$ \sum_{x \in X_n(\kappa, \tau)} \prod_{i \in I(x)} |n^2 - i^2+z|^{-1}
\leq \sum_{i_1,.., i_\rho \neq \pm n} \frac{2^\rho}{|n^2
-i_1^2|\cdots |n^2 - i_\rho^2|} $$ $$ =\left ( \sum_{i \neq \pm n}
\frac{2}{|n^2 - i^2|}  \right )^\rho  \leq \left ( \frac{C \log n
}{n}  \right )^\rho, \quad  \rho = (s+1)\kappa,$$ because $\sum_{i
\neq \pm n} |n^2 - i^2|^{-1} \leq (C\log n)/n $ (e.g., see Lemma 10
in \cite{DM10}). Therefore, taking a sum over $ \tau = 1, \ldots ,
m+1 $ in (\ref{4.14.5}), we obtain the following.
\begin{Lemma}
\label{lem14.1} In the above notations, for $\kappa= 1,2, \ldots,$
\begin{equation}
\label{4.14.6} \sum_{x \in X_n(\kappa)} |h_1 (x,z)| \leq  C
n^{1-\frac{1}{s}} \left ( \frac{C \log n }{n}  \right )^{(s+1)\kappa}
 \left ( H^- + H^+  \right ).
\end{equation}
\end{Lemma}
\bigskip

4. Now we are going to show that the main term of the asymptotics of
$\beta_n^+ (z), \;|z|\leq 1, $ is given by $H^+ - H^-.$ First we
prove the following.
\begin{Lemma}
\label{lem4.3} In the above notations, for $n=sm-1,$ we have
\begin{equation}
\label{eq4.3} \sum_{x \in X_n\setminus X_n(0)} |h (x,z)|  \leq C(a,b)
\frac{\log n}{n^{1/s}} \left ( \frac{\log n }{n}  \right )^s  (H^-(m)
+ H^+ (m)).
\end{equation}
\end{Lemma}

\begin{proof}
If $x \in X_n(\kappa), $ then $\nu (x) = p+ q $ with $p = 1 + s
\kappa,$ and $ q = m+ \kappa, $ so
$$
h(x,z) = a^p b^q h_1 (x,z) = ab^m (a^s b)^\kappa  h_1 (x,z).
$$
By (\ref{4.14.6}) it follows
$$
\sum_{x \in X_n(\kappa)} |h (x,z)| \leq  C |a||b|^m n^{1-\frac{1}{s}}
(D (n))^\kappa
 \left ( H^-(m) + H^+(m)  \right ),
$$
with $ D (n) = |a^s b| \left ( \frac{C \log n }{n}  \right )^{s+1}. $
For large enough $n$ we have $ D (n) < 1/2, $ so
$$
\sum_{\kappa=1}^\infty \sum_{x \in X_n(\kappa)} |h (x,z)| \leq C
|a||b|^m n^{1-\frac{1}{s}} D (n)
 \left ( H^-(m) + H^+(m)  \right ),
 $$
which completes the proof.

\end{proof}

Since
$$ \beta_n^+ (z) = \sum_{x \in X_n(0)} h (x,z)+ \sum_{x \in X_n\setminus X_n(0)} h
(x,z),$$  Lemma~\ref{lem4.1} and Lemma~\ref{lem4.3} lead to the
following.
\begin{Proposition}  In the above notations,
for $n=sm-1,$ we have
\begin{equation}
\label{4.18.1} \beta_n^+ (z) = \beta_n^+ (0) [1+ O ((\log n)/n)],
\end{equation}
where
\begin{equation}
\label{4.18.2} \beta_n^+ (0) = \frac{-2sa b^m}{(2s)^{2m} m!}\,
\frac{\Gamma^2 (1-\frac{1}{s})\Gamma (m-\frac{2}{s})}{\Gamma^2
(m-\frac{1}{s})\Gamma (1-\frac{2}{s})} \left ( 1 +
  O \left ( (\log n)^{s+1}/n^s \right )    \right ).
\end{equation}
\end{Proposition}

\begin{proof}
Indeed,  by (\ref{4.8.4}) one can easily see that  (\ref{4.18.1})
follows from (\ref{49.4}) and (\ref{eq4.3}).

To prove (\ref{4.18.2}), let us recall that
$$
\sum_{\xi \in X_n(0)} h(\xi,0) = H^+ (m) - H^-(m),
$$
so (\ref{eq4.3})  and (\ref{4.8.4}) imply that
\begin{equation}
\label{4.18.3} \beta_n^+ (0) = a b^m (H^+(m) - H^-(m) ) \left ( 1 +
  O \left ( (\log n)^{s+1}/n^s \right )    \right ).
\end{equation}
Therefore,  (\ref{4.18.2}) follows from (\ref{4.4.5a}) and
(\ref{4.7.2}), which completes the proof.
\end{proof}

\bigskip

5. Next we estimate $\beta_n^-(z) $ for $ |z|\leq 1$  -- compare
Lemma \ref{lem3.2} -- without any restriction like (\ref{3.4.4}) or
(\ref{4.1.1}) on $n.$ For every $n,$ if $y$ is a path from $+n$ to
$-n$ satisfying (\ref{4.2.2}) - (\ref{4.2.3}), then we have $- 2n
=-2p +2sq,$ i.e.,
\begin{equation}
\label{4.40} p = n+ sq.
\end{equation}
We define $Y_n(q) $ as the set of all paths (\ref{4.2.2}) with
parameters $p,q $ satisfying (\ref{4.40}). Then
\begin{equation}
\label{4.41} \# Y_n(0) = 1
\end{equation}
and the only path $\eta \in Y_n(0) $ is defined by
\begin{equation}
\label{4.42} \eta (t) = -2, \quad 1\leq t \leq n,
\end{equation}
so its vertices are
\begin{equation}
\label{4.43} j(t;\eta)= n-2t, \quad 0 \leq t \leq n.
\end{equation}
Therefore,
\begin{equation}
\label{4.44} h_1 (\eta) = \prod_{t=1}^{n-1} [n^2 - (n-2t)^2]^{-1}=
\frac{1}{4^{n-1} [(n-1)!]^2}
\end{equation}
and, due to Lemma \ref{lem3.0},
\begin{equation}
\label{4.440} h_1 (\eta,z) = \prod_{t=1}^{n-1} [n^2 -
(n-2t)^2+z]^{-1}= h_1 (\eta) \, [1+O((\log n)/n)].
\end{equation}

If $q \geq 1,$ then any path $y \in Y_n(q) $ has a sub-path with $s+1
$ steps of the form $(2s, -2, \ldots, -2).$  Indeed, choose
\begin{equation}
\label{4.45} t^* = \max\{t: \; y(t)= 2s\};
\end{equation}
then $t^* \leq \nu (y) -s-1,$ and
\begin{equation}
\label{4.46} y(t)= -2, \quad t^*+1 \leq t \leq t^* +s.
\end{equation}
Now define a new path $\tilde{y} \in Y(q-1)$ by
\begin{equation}
\label{4.47}
\tilde{y}(t) = \begin{cases} y(t),  &  1 \leq t < t^*,\\
y(t+1+s), &   t^* \leq t\leq \nu (y) -s.
\end{cases}
\end{equation}
Then
\begin{equation}
\label{4.48} h_1 (y,z) = h_1 (\tilde{y},z) \cdot \prod_{t=t^*}^{t^*
+s} [n^2 - (n-j^2 (t,y))^2+z]^{-1},
\end{equation}
so
$$
|h_1 (y,z) | \leq (2n)^{-(s+1)} |h_1 (\tilde{y},z)| \quad \text{for}
\;\; |z|\leq 1.
$$
After $q$ such restructuring we come, in view of (\ref{4.440}), to
the inequality
\begin{equation}
\label{4.50} |h_1 (y,z) | \leq 2(2n)^{-q(s+1)} |h_1 (\eta)|, \quad
|z|\leq 1, \; n>N_1.
\end{equation}
If $T=\max\{|a|, |b|\},$ then -- compare (\ref{3.8.4}) -
(\ref{3.9.1}) -- for $y\in Y_n(q) $ it follows from (\ref{4.50}) that
\begin{equation}
\label{4.51} |h(y,z)| = |a^{n+qs} b^q||h_1 (y,z)| \leq
\frac{2T^{q(s+1)}}{(2n)^{q(s+1)}} |a^nh_1 (\eta)| = 2|h(\eta)| \left
(\frac{T}{2n}\right )^{q(s+1)}.
\end{equation}

As in (\ref{3.10.1}), now we can claim that
\begin{equation}
\label{4.52} \# Y_n(q) \leq \begin{pmatrix}  p+q\\q \end{pmatrix} =
\begin{pmatrix}  n+q(s+1)\\ q \end{pmatrix}  \leq \begin{cases}
\frac{1}{q!}(s+2)^q n^q &   \text{if} \; q < n,\\
2^{(s+2)q}  &   \text{if} \; q \geq n.
\end{cases}
\end{equation}

Therefore, by (\ref{4.51}) and (\ref{4.52}) we obtain
\begin{equation}
\label{4.53} \sum_{q\geq 1}\sum_{y\in Y_n(q)} |h(y,z)| \leq
2|h(\eta)| (\sigma_1 + \sigma_2),
\end{equation}
where for large enough $n$
\begin{equation}
\label{4.54} \sigma_1 = \sum_{q=1}^{n-1} \frac{(s+2)^q n^q}{q!}
\left ( \frac{T}{2n} \right )^{q(s+1)} \leq C_1 n^{-s},
\end{equation}
and
\begin{equation}
\label{4.55} \sigma_2 = \sum_{q=n}^{\infty} 2^{q(s+2)}\left (
\frac{T}{2n} \right )^{q(s+1)} \leq 2^q \left ( \frac{T}{n} \right
)^{n(s+1)}.
\end{equation}
Certainly, the inequalities (\ref{4.52}) - (\ref{4.55}) imply
\begin{equation}
\label{4.56} \sum_{Y_n\setminus \{\eta\}} |h(y,z)| \leq
\frac{C}{n^{s}}|h(\eta)|, \quad |z|\leq 1.
\end{equation}

\begin{Proposition}
\label{prop14} In the above notations,
\begin{equation}
\label{4.57}  \beta_n^- (z)  = \beta_n^- (0) (1+O((\log n)/n)),
\end{equation}
where
\begin{equation}
\label{4.58} \beta^-_n (0)  = \frac{a^n}{4^{n-1} [(n-1)!]^2}
(1+O(1/n^s)).
\end{equation}
\end{Proposition}

\begin{proof}
Indeed, (\ref{4.58}) follows from (\ref{4.44}) and (\ref{4.56}), and
(\ref{4.57}) follows from (\ref{4.44}), (\ref{4.58}), (\ref{4.440})
and (\ref{4.56}).
\end{proof}

\begin{Theorem}
\label{thm5} For any potential of the form
\begin{equation}
v(x) = ae^{-2ix}+be^{2isx},\quad  a,b \neq 0, \; s\geq 3,
\end{equation}
there is no basis consisting of root functions of $L_{Per^-}(v).$
\end{Theorem}

\begin{proof}
In view of (\ref{4.18.1}), (\ref{4.18.2}) and (\ref{4.57}), we may
apply Criterion~\ref{crit1} to the set $\Delta= \{n= sm-1, \; m \in
\mathbb{N}\}.$     By (\ref{4.18.2}), (\ref{4.57}) and the Stirling
formula, we have
\begin{equation}
\label{4.59} |\beta^-_n (0)|/|\beta^+_n (0)|  \leq C_1^n \left (
m!/n! \right )^2  \leq C_2^m m^{2(1-s)m} \to 0, \quad n \in \Delta.
\end{equation}
Hence, Criterion~\ref{crit1} implies that there is no basis
consisting of root functions of $L_{Per^-}(v).$
\end{proof}

\section{Comments}

Theorems \ref{thm3.1} and \ref{thm5} claim divergence of spectral
decompositions in the case of potentials of the form
\begin{equation}
\label{5.1} v(x) = a e^{-2iRx} + b e^{2iSx}
\end{equation}
for many pairs $R, S$ such that $R\neq S.$

If $R=S$  the picture is much simpler; it is similar to the case
$R=S=1$ which is analyzed in  \cite{DM25}, see Theorem~7 in
Section~3 there.

If $R=S>1,$ then an admissible path $x$ from $-n$ to $n $ (or from
$n$ to $-n$) gives a nonzero term $h(x,z)$ of $\beta^+_n (z)$ if and
only if $x(t)= \pm 2R.$ Let $p$  and $q $ be, respectively, the
number of steps equal to $-2R$  and $2R.$ Then -- compare
(\ref{3.3.2}) - (\ref{3.4.6}) --
\begin{equation}
\label{5.2} 2n= - 2Rp+2Rq = 2R(p+q),
\end{equation}
so
\begin{equation}
\label{5.3} \beta_n^- (z) =0, \quad  \beta_n^+ (z) =0 \quad
\text{if} \quad n\not \equiv 0 \mod R,
\end{equation}
Choose $N $ so large that (\ref{2.7}) holds and the claim of
Lemma~\ref{lem1} is valid for $n>N.$  Set
\begin{equation}
\label{5.4}   \Delta_0^\pm = \{n\in \Gamma^\pm: n>N, \;  n\not \equiv
0 \mod R\}
\end{equation}
and let $E(\Delta_0^\pm)= Ran \, P_{\Delta_0^\pm},$ where $P_\Delta$
is the projection defined by (\ref{3.1.1})). Then, in view of
(\ref{5.3}), Criterion~\ref{crit1} implies that $E(\Delta^+_0) $
(respectively $E(\Delta^-_0) $)  has a basis consisting of periodic
(antiperiodic) root functions. In particular, this holds for the set
$\Delta_0 $ defined by (\ref{3.3.1}).

On the other hand, let us consider the set $$\Delta^\pm_1=\{n\in
\Gamma^\pm, \; n=R\,m, \; m\geq N \}.$$ By Criterion~\ref{crit1} the
system of root functions of $L_{Per^\pm} $ contains a (Riesz) basis
in $L^2([0,\pi])$  if and only if $E(\Delta^\pm_1)$ has a basis
consisting of periodic (respectively antiperiodic) root functions.

One can show using the same argument as in \cite[Section 3]{DM25}
(see Lemmas 3 and 4, and Propositions 5 and 6 there) that if $n \in
\Delta_1^\pm, $ then
\begin{equation}
\label{5.5} \beta_n^+ (z) = 4R^2 \left ( \frac{b}{4R^2}\right )^m
\frac{1}{[(m-1)!]^2} \left ( 1+O \left (\frac{\log n}{n} \right )
\right ), \;\; |z| \leq 1,
\end{equation}
\begin{equation}
\label{5.6} \beta_n^- (z) = 4R^2 \left ( \frac{a}{4R^2}\right )^m
\frac{1}{[(m-1)!]^2} \left ( 1+O \left (\frac{\log n}{n} \right )
\right ), \; \; |z| \leq 1.
\end{equation}
Now Criterion~\ref{crit1} says when  $E(\Delta_1^\pm) $ has a basis
consisting of root functions, which leads to the following
generalization of Theorem~7 in \cite{DM25}.
\begin{Proposition}
\label{prop20} If $R$ is even, then a root function system of the
operator
\begin{equation}
\label{5.7}   L = -\frac{d^2}{dx^2} + ae^{-2iRx} + b e^{2iRx},
\end{equation}
considered with antiperiodic boundary conditions, contains a Riesz
basis in $L^2 ([0,\pi]).$

If $R$ is odd and $L$ is considered with antiperiodic boundary
conditions, or $R$ is arbitrary and $L$ is considered with periodic
boundary conditions, then the system of root functions of the
operator $L$ contains a Riesz basis in $L^2 ([0,\pi])$ if and only if
\begin{equation}
\label{5.8} |a|=|b|.
\end{equation}
\end{Proposition}

\begin{proof}
By (\ref{5.5}) and (\ref{5.6}) we have
$$
\frac{\beta^-_n (z)}{\beta^+_n (z)} = \left (\frac{a}{b} \right )^n
\left ( 1+O \left (\frac{\log n}{n} \right ) \right ), \quad n\in
\Delta_1^\pm, \;\; |z| \leq 1.
$$
Then the assertion follows from the simple observation that
$\Delta_1^+ \cap 2\mathbb{N}$ is an infinite set for any $R$ but
$\Delta_1^- \cap (2\mathbb{N}-1)= \emptyset $ if $R$ is even and
$\Delta_1^- \cap (2\mathbb{N}-1) $ is infinite if $R$ is odd.
\end{proof}

\end{document}